\documentclass[english,a4paper]{amsart}
\usepackage[T1]{fontenc}
\usepackage[latin1]{inputenc}
\usepackage{graphicx}
\usepackage{amssymb}
\usepackage{amsbsy}
\usepackage{bbm}
\usepackage{setspace}

\makeatletter
 \theoremstyle{plain}
\newtheorem{thm}{Theorem}[section]
  \theoremstyle{plain}
  \newtheorem{prop}[thm]{Proposition}
  \theoremstyle{plain}
  \newtheorem{lem}[thm]{Lemma}
 \theoremstyle{definition}
 \newtheorem*{defn*}{Definition}
  \theoremstyle{plain}
  
  \theoremstyle{remark}
    \newtheorem{rem}[thm]{Remark}

\renewcommand{\hat}{\widehat}
\renewcommand{\phi}{\varphi}

\DeclareMathOperator{\diam}{diam}

\newcommand{\R}{\mathbb {R}}

\newcommand{\N}{\mathbb {N}}

\newcommand{\1}{\mathbbm{1}}

\renewcommand{\hat}{\widehat}
\renewcommand{\phi}{\varphi}
\renewcommand{\epsilon}{\varepsilon}
 
\def\diam {\mathop {\hbox{\rm diam}}}

\def\R{\mathbb{R}}

\def\N{\mathbb{N}}

\usepackage{mathptmx}

\usepackage{microtype}

\usepackage{babel}
\makeatother

\begin{document}
 \onehalfspacing
 

\title[On the Lebesgue 
measure of sum-level sets]{On the 
 Lebesgue measure of
sum-level sets for continued fractions}

\author{Marc Kesseböhmer}

\address{Fachbereich 3 -- Mathematik und Informatik, Universität Bremen, Bibliothekstr.
1, D--28359 Bremen, Germany}

\email{mhk@math.uni-bremen.de}

\author{Bernd O. Stratmann}

\address{Mathematical Institute, University of St. Andrews, North Haugh, St.
Andrews KY16 9SS, Scotland}

\email{bos@maths.st-and.ac.uk}


\date{\today}

\begin{abstract}\singlespacing
In this paper we give a detailed measure theoretical analysis of 
what we call sum-level sets for regular continued fraction 
expansions. The first main result is to settle a recent conjecture of 
Fiala and Kleban, which asserts that the Lebesgue measure 
of these level sets decays to zero, for the level tending to 
infinity.  The second and third main result then give  precise asymptotic 
estimates for this decay. The proofs of these results are based on recent 
progress in   
infinite ergodic theory, and in particular, they  give  non-trivial
applications of this theory to number theory. The paper closes with a discussion of the 
thermodynamical significance of the obtained results, and with some 
applications of these to metrical Diophantine analysis.
\end{abstract}
\maketitle
\section{Introduction and statements of result}

\noindent In this paper we consider classical number theoretical dynamical
systems arising from the Gauss map $\mathfrak{g}:x\mapsto1/x\mod1$
(for $x\in[0,1]$). It is well known that the inverse branches of
$\mathfrak{g}$ give rise to an expansion of the reals in the unit
interval with respect to the infinite alphabet $\N$. This expansion
is given by the regular continued fraction expansion \[
[a_{1},a_{2},\ldots]:=\cfrac{1}{a_{1}+\cfrac{1}{a_{2}+\ldots}},\]
 where all the $a_{i}$ are positive integers.
 
 The main task of this paper is to give a detailed measure-theoretical
analysis of the following sets $\mathcal{C}_{n}$, for $n \in \N$, which 
we will refer
to as the sum-level sets: \[
\mathcal{C}_{n}:=\{[a_{1},a_{2},\ldots] \in [0,1]:\sum_{i=1}^{k} a_{i}=n 
\mbox{ for some } k\in\N\}.\]
 A first inspection of the sequence of these sets shows that
 $\liminf_{n} \mathcal{C}_{n}$
is equal to the set of all noble numbers, that is, numbers whose infinite
continued fraction expansions end with an infinite block of $1$'s.
Also, one immediately verifies that $\limsup_{n}\mathcal{C}_{n}$
is equal to the set of all irrational numbers in $[0,1]$. Hence,
at first sight, the sequence of sum-level sets appears to be
far away from being a canonical dynamical entity. In order to state
the main results, 
 note that for the first four members of the sequence of the sum-level
sets (cf. Fig. \ref{fig:first-level-sum-sets}) one immediately computes that \[
\lambda(\mathcal{C}_{1})=1/2,\lambda(\mathcal{C}_{2})=1/3,\lambda(\mathcal{C}_{3})=3/10,\lambda(\mathcal{C}_{4})=39/140.\]
 From this one might already suspect that $\lambda\left(\mathcal{C}_{n}\right)$
is decreasing for $n$ tending to infinity. In fact, it was conjectured
by Fiala and Kleban in \cite{FK} that $\lambda\left(\mathcal{C}_{n}
\right)$
tends to zero, as $n$ tends to infinity. The first  main result
of this paper is to settle this conjecture.
\begin{thm}
\label{thm1} \[
\lim_{n\to\infty}\lambda(\mathcal{C}_{n})=0.\]
\end{thm}
\noindent We  give two independent proofs of this theorem. The first 
of these is almost elementary and 
only
 mildly spiced with infinite ergodic
theory, whereas the second proof  will be deduced from  a
significantly  stronger 
result (see Proposition \ref{pre-exact} for the details). In a 
nutshell, here we give a detailed proof of the fact that the
Farey map $T$ is an exact transformation, which in turn 
allows to use a criterion of Lin in order 
to deduce the result. 

For the next station on our journey of investigating the asymptotic 
behaviour of the sequence $\left(\lambda(\mathcal{C}_{n})\right)$,
we  employ the continued fraction mixing property of the induced map 
of the Farey map $T$ on 
$\lambda(\mathcal{C}_{1})$, in order to show that $\mathcal{C}_{1}$ is 
a Darling--Kac set for $T$. A computation of the return sequence 
of $T$  then leads to the following theorem, 
\begin{figure}
\includegraphics[width=0.8\textwidth]{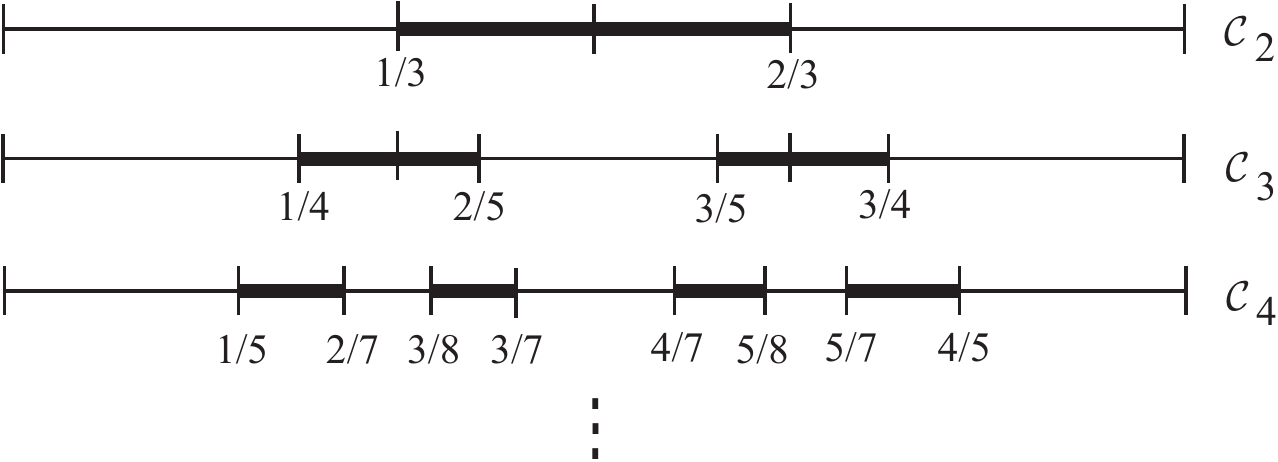}
\caption{\label{fig:first-level-sum-sets}The first sum-level sets.}
\end{figure}
where we use the 
common notation $b_{n}\sim c_{n}$
to denote that $\lim_{n\to\infty}b_{n}/c_{n}=1$. 
\begin{thm}\label{thm1.5}
\[
\sum_{k=1}^{n}\lambda\left(\mathcal{C}_{k}\right)
\sim\frac{n}{\log_{2} n}.\]
 \end{thm}
 Our third theorem gives a significant improvement of Theorem
\ref{thm1} and Theorem
\ref{thm1.5}. That is, by increasing the dosage of infinite ergodic
theory, we obtain the following sharp estimate for the asymptotic
behaviour of the Lebesgue measure of the sum-level sets. 
\begin{thm}
\label{thm2} \[
\lambda(\mathcal{C}_{n})\sim\frac{1}{\log_{2}n}.\]
\end{thm}
We then continue by relating these results on the sum-level sets to the thermodynamical analysis of the
Stern--Brocot system obtained in \cite{KesseboehmerStratmann:07}. We obtain the, on a first sight, slightly surprising  result  that this thermodynamical analysis  can 
be obtained from an exclusive use of either the sequence $\left(\mathcal{C}_{n}\right)$ or 
alternatively its complementary sequence $\left(\mathcal{C}_{n}^{c}\right)$, rather than 
using the Stern--Brocot sequence in total. 
 In particular, this reveals that the vanishing of $\lim_{n\to\infty}\lambda(\mathcal{C}_{n})$
is very much a phenomenon of the fact that the Stern--Brocot system
has a phase transition of order two at the point at which infinite ergodic
theory takes over the regime from finite ergodic theory.
A detailed discussion of this application to the thermodynamical 
formalism  is given in  Section \ref{Thermo}. Finally, 
in Section \ref{Diophant},
we apply Theorem \ref{thm2} to classical metrical Diophantine analysis, 
and  derive in this way a certain 
algebraic Khintchine-like law  (see Lemma \ref{CorD}).

\section{Sum-level sets, Stern--Brocot intervals, and the infinite
Farey system\label{subsec:Farey}}

In the introduction we defined the sequence $\left(\mathcal{C}_{n}\right)$
of sum-level sets via the sum of the first  entries in the continued 
fraction expansions. For later
convenience, let us also add $\mathcal{C}_{0}:=[0,1]$ to this sequence.
Let us begin with some brief comments on various equivalent ways of expressing the sum-level
sets. 

\subsection{ $\pmb{\mathcal{C}_{n}}$ in terms of Stern--Brocot intervals}  
Recall the following classical construction
of Stern--Brocot intervals (SB--intervals) (cf. \cite{Stern1858},
\cite{Brocot:1860}). For each $n\in\N_{0}$, the elements of the
$n$-th member of the Stern--Brocot sequence \[
\left\{ \frac{s_{n,k}}{t_{n,k}}:k=1,\ldots,2^{n}+1\right\} \]
 are defined recursively as follows: 
\begin{itemize}
\item $s_{0,1}:=0\,\,$ and $\,\, s_{0,2}:=t_{0,1}:=t_{0,2}:=1$; 
\item $s_{n+1,2k-1}:=s_{n,k}\quad\textrm{and}\quad t_{n+1,2k-1}:=t_{n,k},$
for $k=1,\ldots,2^{n}+1$; 
\item $s_{n+1,2k}:=s_{n,k}+s_{n,k+1}\quad\textrm{and}\quad t_{n+1,2k}:=t_{n,k}+t_{n,k+1}$,
for $k=1,\ldots2^{n}$. 
\end{itemize}
The set $\mathcal{T}_{n}$ of SB--intervals of order $n$ is 
given by \[
\mathcal{T}_{n}:=\left\{ \left[\frac{s_{n,k}}{t_{n,k}},\frac{s_{n,k+1}}{t_{n,k+1}}\right]:\, k=1,\ldots,2^{n}\right\} .\]
 By means of these intervals, the sum-level sets $\mathcal{C}_{n}$
are then given as follows. For $n=0,1$, we have $\mathcal{C}_{0}=[s_{0,1}/t_{0,1},s_{0,2}/t_{0,2}]$
and $\mathcal{C}_{1}=[s_{1,2}/t_{1,2},s_{1,3}/t_{1,3}]$. For $n>1$,
we have \[
\mathcal{C}_{n}=\bigcup_{k=1}^{2^{n-2}}\left[\frac{s_{n,4k-2}}{t_{n,4k-2}},\frac{s_{n,4k}}{t_{n,4k}}\right].\]
 Note that this point of view of $\mathcal{C}_{n}$ is the one chosen
in \cite{FK}, where $\mathcal{C}_{n}$ was referred to as  the set
of even intervals.  Also, note that these even intervals are not SB--intervals.
However, we clearly have that each of them is the union of two neighbouring
SB--intervals of order $n$. That is, \[
\left[\frac{s_{n,4k-2}}{t_{n,4k-2}},\frac{s_{n,4k}}{t_{n,4k}}\right]
=\left[\frac{s_{n,4k-2}}{t_{n,4k-2}},\frac{s_{n,4k-1}}{t_{n,4k-1}}
\right]\cup\left[\frac{s_{n,4k-1}}{t_{n,4k-1}},\frac{s_{n,4k}}{t_{n,4k}}
\right].\]
Throughout, we will use the notation $\mathcal{C}_{n}^{c}$ to  
denote the set of SB--intervals of order $n$ that are not in 
$\mathcal{C}_{n}$. Also, by slight abuse of notation, occasionally  we will write 
$I \in \mathcal{C}_{n}$ for a SB--interval $I \in \mathcal{T}_{n}$  
which is a subset of $\mathcal{C}_{n}$. 

\subsection{$\pmb{\mathcal{C}_{n}}$ in terms of  Stern--Brocot coding} 
There is also a way of expressing the sequence
$\left(\mathcal{C}_{n}\right)$
in terms of the maps $\alpha,\beta:  \mathcal{C}_{0} \to \mathcal{C}_{0}$
given by
\[ \alpha(x):=  x/(1+x) \mbox{ and } \beta\left(x\right):=
1/(2-x).\]
 It is well known that the orbit of the unit interval under the free
semi-group generated by
$\alpha$ and $\beta$ is in 1--1 correspondence
to the set of SB-intervals. 
In fact, by associating the symbol $A$
to the map $\alpha$ and the symbol $B$ to the map $\beta,$
one obtains that each SB-interval (with the exception the SB-interval
of order $0$) is associated with a unique word made of letters from
the alphabet $\{A,B\},$ and vice versa. We will refer to this coding
as the Stern--Brocot coding, and will write $I\cong W$ if $I$ is
the SB-interval whose Stern--Brocot code is given by $W\in\{A,B\}^{k}$,
for some $k\in\N$. The reader might like to recall that there is
a dictionary  which translates between Stern--Brocot intervals and 
continued fraction cylinder sets
$[\hspace{-.6mm}[a_{1},\ldots,a_{n}]\hspace{-.6mm}]:=\{[ x_{1},x_{2}, \ldots 
   ]: x_{k}=a_{k}, k=1,\ldots,n\}$, which reads as follows. 
For $\{X,Y\}=\{U,V\}=\{A,B\},$
we have 
\[ X^{a_{1}}Y^{a_{2}}X^{a_{3}}\cdots U^{a_{k}}V\cong
\left\{ \begin{array}{lll}
[\hspace{-.6mm}[a_{1}+1,a_{2},a_{3},\ldots,a_{k}]\hspace{-.6mm}] & \,\,\textrm{for} & X=A\\
{}[\hspace{-.6mm}[1,a_{1},a_{2},\ldots,a_{k}]\hspace{-.6mm}] & \,\,\textrm{for} & X=B.\end{array}\right.\]
 By using this dictionary, it is not hard to see that for $n\geq2$
we have 
\[
\mathcal{C}_{n}=\{I\in \mathcal{T}_{n}:I\cong WXY\mbox{ for $\{X,Y\}=\{A,B\}$ and $W\in\{A,B\}^{n-2}$}\}.\]
 To illustrate this way of viewing $\mathcal{C}_{n}$, we list
the first members of this sequence of code words: \begin{eqnarray*}
\mathcal{C}_{1}: & B\\
\mathcal{C}_{2}: & AB\,\,\,\, BA\\
\mathcal{C}_{3}: & AAB\,\,\,\, ABA\,\,\,\, BAB\,\,\,\, BBA\\
\mathcal{C}_{4}: & AAAB\,\,\,\, AABA\,\,\,\, ABAB\,\,\,\, ABBA\,\,\,\, BAAB\,\,\,\, BABA\,\,\,\, BBAB\,\,\,\, BBBA\\
 & \vdots\end{eqnarray*}
 \subsection{$\pmb{\mathcal{C}_{n}}$ in terms of the Farey map} 
 The sequence
 $\left(\mathcal{C}_{n}\right)$ can also be expressed
 with the help of the Farey map $T:\mathcal{C}_{0}\rightarrow\mathcal{C}_{0}$.
 For this, recall that $T$ is given by\[
 T\left(x\right):=\left\{ \begin{array}{lll}
 x/(1-x) & \,\,\textrm{for} & x\in\left[0,1/2\right]\\
 (1-x)/x & \,\,\textrm{for} & x\in\left(1/2,1\right],\end{array}\right.\]
  and that the inverse branches of $T$ are given by\[
 u_{0}\left(x\right):=x/(1+x)\mbox{ and }u_{1}\left(x\right):=
 1/(1+x).\]
The associated Markov partition is then given by $\{L,R\}$, 
where $L:= 
\mathcal{C}_{0} \setminus \mathcal{C}_{1}$ and 
$R:=\mathcal{C}_{1}$,
and each irrational number in $\mathcal{C}_{0}$ has a
Markov coding 
$x=\langle x_{1}, x_{2},\ldots \rangle \in \{L,R\}^{\N}$, 
given by $T^{k-1}(x)\in x_{k}$
for all $k \in \N$.
This coding will be referred to
 as the Farey coding, and will write $I \triangleq W$ if $I$ is
 the SB-interval whose Farey code is given by $W\in\{L,R\}^{k}$,
 for some $k\in\N$. The  dictionary  which translates between Farey codes and 
 continued fraction cylinders reads as follows:
 \[L^{a_{1}-1}RL^{a_{2}-1}RL^{a_{3}-1} \cdots L^{a_{k}-1}R \triangleq
 [\hspace{-.6mm}[a_{1},a_{2},a_{3},\ldots,a_{k}]\hspace{-.6mm}] .
 \]
  By using this dictionary, it is not hard to 
  see that we have, for each $n\in \N$,
   \[
 \mathcal{C}_{n}=\{I\in \mathcal{T}_{n}: I \triangleq  WR\mbox{ for 
  $W\in\{L,R\}^{n-1}$}\}.\]
  Again, let us list
 the first members of this sequence of code words: \begin{eqnarray*}
 \mathcal{C}_{1}: & R\\
 \mathcal{C}_{2}: & LR\,\,\,\, RR\\
 \mathcal{C}_{3}: & LLR\,\,\,\, LRR\,\,\,\, RLR\,\,\,\, RRR\\
 \mathcal{C}_{4}: & LLLR\,\,\,\, LLRR\,\,\,\, LRRR\,\,\,\, LRLR\,\,\,\, 
 RRLR\,\,\,\, RRRR\,\,\,\, RLRR\,\,\,\, RLLR\\
  & \vdots\end{eqnarray*}
  \\ 
  The crucial link between the sequence of sum-level sets and the 
  Farey map is now given by the following lemma.
  \begin{lem}\label{Farey}
      For all $n\in\N$, we have that 
      \[T^{-(n-1)}(\mathcal{C}_{1})=\mathcal{C}_{n}.\]
      \end{lem}
      \begin{proof}
By computing the images of $\mathcal{C}_{1}$ under $u_{0}$
	  and $u_{1}$, one immediately verifies that 
	  $T^{-1}(\mathcal{C}_{1})=\mathcal{C}_{2}$.
We then proceed by way of induction as follows.
Assume that for some $n \in \N$ we have that
$T^{-(n-1)}(\mathcal{C}_{1})=\mathcal{C}_{n}$.
Since $T^{-n}(\mathcal{C}_{1})=
T^{-1}(T^{-(n-1)}(\mathcal{C}_{1}))= T^{-1}( \mathcal{C}_{n})$, it 
is then sufficient to show that $T^{-1}(\mathcal{C}_{n})=
\mathcal{C}_{n+1}$. For this, let $x=[a_{1},a_{2},\ldots] \in 
\mathcal{C}_{n}$
be given. Then there exists $\ell \in \N$ such that
$x \in [\hspace{-.6mm}[a_{1},\ldots
,a_{\ell}]\hspace{-.6mm}]$ and $\sum_{i=1}^{\ell} 
a_{i} =n$. By computing the images of $x$ under $u_{0}$
and $u_{1}$, one immediately obtains that $T^{-1}(x) =\{[1,a_{1},a_{2},\ldots],
[a_{1}+1,a_{2},\ldots]\}$. Clearly, since $1+\sum_{i=1}^{\ell} 
a_{i} =(a_{1}+1)+ \sum_{i=2}^{\ell} 
a_{i} =n+1$, this shows that $T^{-1}(x)  \subset \mathcal{C}_{n+1}$, 
and hence, $T^{-1}(\mathcal{C}_{n})\subset
\mathcal{C}_{n+1}$.  The reverse inclusion 
$\mathcal{C}_{n+1} \subset T^{-1}(\mathcal{C}_{n})$ follows for instance  by counting the SB--intervals in $ \mathcal{C}_{n+1}$ and using the dictionary translating between SB--intervals  and 
 continued fraction cylinder sets.
\end{proof}
\subsection{Elementary ergodic theory for the Farey map} 
 For later use we now recall a few elementary facts and results from infinite
ergodic theory for the Farey map. It is well known that
the infinite Farey system $\left(\mathcal{C}_{0},T,\mathcal{A},
\mu\right)$
is a conservative ergodic measure preserving dynamical system. Here,
$\mathcal{A}$ refers to the Borel $\sigma$-algebra of $\mathcal{C}_{0}$,
and the measure $\mu$ is the infinite $\sigma$-finite $T$-invariant
measure absolutely continuous with respect to the Lebesgue measure
$\lambda$. In
fact, with $\phi_{0}:\mathcal{C}_{0}\to\mathcal{C}_{0}$ defined by
$\phi_{0}(x):=x$, it is well known that  $\mu$ is explicitly given 
by  (see e.g. \cite{dan}, \cite{par1} , \cite{par2})\[
\mathrm{d}\lambda=\phi_{0}\,\mathrm{d}\mu.\]
Recall that conservative and ergodic  means that  $\sum_{n\geq0}\hat{T}^{n}
\left(f\right)=\infty$, $\mu$-almost everywhere and for
all $f\in L_{1}^{+}\left(\mu\right):=\left\{ f\in L_{1}
\left(\mu\right):\; f\geq0\; 
\mathrm{and}\;\mu(f\cdot 
\1_{\mathcal{C}_{0}})>0\right\} $.
Here, $\1_{\mathcal{C}_{0}}$ refers to the characteristic function
of $\mathcal{C}_{0}$. Also,  invariance of $\mu$ under $T$ means
$\hat{T}\left(\1_{\mathcal{C}_{0}}\right)=\1_{\mathcal{C}_{0}}$,
where $\hat{T}:L_{1}\left(\mu\right)\to L_{1}\left(\mu\right)$ denotes the transfer operator 
associated with the infinite dynamical Farey system, which 
is a positive
linear operator,
given by \[
\mu\left(\1_{C}\cdot\hat{T}\left(f\right)\right)=\mu\left(\1_{T^{-1}\left(C\right)}\cdot f\right),\mbox{ for all }f\in L_{1}\left(\mu\right),C\in\mathcal{A}.\]
 Finally, note that the Perron--Frobenius operator $\mathcal{L}:L_{1}\left(\mu\right)\to L_{1}\left(\mu\right)$
of the Farey system is given by \[
\mathcal{L}\left(f\right)=\left|u_{0}'\right|\cdot(f\circ u_{0})+\left|u_{1}'\right|\cdot(f\circ u_{1}),\mbox{ for all }f\in L_{1}\left(\mu\right).\]
 One then immediately verifies that the two operators $\hat{T}$ and
$\mathcal{L}$ are related as follows: \[
\hat{T}\left(f\right)=\phi_{0}\cdot\mathcal{L}\left(f/\phi_{0}\right),\mbox{ for all }f\in L_{1}\left(\mu\right).\]

\begin{rem}
\label{remC} Let us remark that $\left(\mathcal{C}_{n}\right)$
has the following topological self-similarity property. Note that the set of SB--intervals of order $2$ consists of four
SB--intervals, that is, a pair of adjacent intervals in the middle 
whose union is equal to
$\mathcal{C}_{2}$, and two surrounding intervals, one to the left and
the other to the right of this pair, where the union of the latter
two is equal to $\mathcal{C}_{2}^{c}$. This structure of how the
intervals of $\mathcal{C}_{2}$ and $\mathcal{C}_{2}^{c}$ appear
in the set of SB--intervals of order $2$ serves as the building
block  for the topological structure of the appearance of the intervals
in $\mathcal{C}_{n}\cup\mathcal{C}_{n}^{c}$ in general. Namely, for
$n>2$, the set of SB--intervals of order $n$ consists of $2^{n}$
SB--intervals which are grouped into $2^{n-2}$ blocks of four adjacent
intervals. The appearance of the intervals in each of these blocks
looks topologically like a scaled down version of the building block
at $n=2$ (see Figure \ref{fig:first-level-sum-sets}). This point of view will be useful in the proof of Lemma
\ref{Lemma0} below, where we will employ a finite inductive process
in order to locate a certain subset of $\mathcal{C}_{n}^{c}$. 
\end{rem}

\section{Proof(s) of Theorem \ref{thm1}}
In this section we give two alternative proofs of  Theorem \ref{thm1}.
The first of these is more elementary, whereas the second uses exactness 
of $T$ and
a criterion for exactness due to Lin. 
\subsection{First Proof of Theorem \ref{thm1}}
The following lemma gives the first step in our first proof of Theorem \ref{thm1}.
Note that the statement of this lemma has already been obtained in
\cite{FK}, where it was the main result. Nevertheless, in order
to keep the paper as self-contained as possible, we give a short proof
of this result. 
\begin{lem}
\label{Lemma0} \[
\liminf_{n\to\infty} \, \lambda(\mathcal{C}_{n})=0.\]
 \end{lem}
\begin{proof}
Let $n\in\N$ be fixed such that $n>3$, and let $k\in\{2,\ldots,n-2\}$
be arbitrary. Recall that the set of SB--intervals of order $k$ consists
of $2^{k-2}$ blocks of four adjacent SB--intervals (see Remark \ref{remC}).
Now, let $I\subset{\mathcal{C}}_{k}$ be a SB--interval of order $k$
such that $I\cong W\in\{A,B\}^{k}$. We then have that $I$ contains
the interval $I_{A,n-k}\cong WA^{n-k}$ as well as the interval $I_{B,n-k}\cong WB^{n-k}$.
Note that $I_{A,n-k}$ and $I_{B,n-k}$ are two distinct SB--intervals
of order $n$ which are both contained in ${\mathcal{C}}_{n}^{c}$.
Also, it is well known (see e.g. \cite{KesseboehmerStratmann:04B})
that in this situation we have, where $a_{n}\asymp b_{n}$ means that
the quotient $a_{n}/b_{n}$ is uniformly bounded away from zero and
infinity, \[
\lambda(I)\asymp(n-k)\lambda(I_{X,n-k}),\mbox{ for each }X\in\{A,B\}.\]
 Clearly, $I_{X,n-k}\cap J_{Y,n-k}=\emptyset$, for all $X\in\{A,B\}$,
$I,J\in{\mathcal{C}}_{k}$ ($I\neq J$). Moreover, by construction,
we have for each $k,l\in\{2,\ldots,n-2\}$ such that $k\neq l$ and
such that either $k$ and $l$ are both odd or both even, \[
I_{X,n-k}\cap J_{Y,n-l}=\emptyset,\mbox{ for all }I\in{\mathcal{C}}_{k},J\in{\mathcal{C}}_{l},X\in\{A,B\}.\]
 Note that in here we require that $k$ and $l$ are both odd or both
even, since for instance for the interval $I\in{\mathcal{C}}_{2}$
for which $I\cong AB$ and the interval $J\in{\mathcal{C}}_{3}$ for
which $J\cong ABA$ we have that $I_{A,n-2}=J_{A,n-3}$. Also, note
that we require $k<n-1$, since for instance for the interval $I\in{\mathcal{C}}_{n-1}$
for which $I\cong WB$ we have that $I_{A,1}\notin{\mathcal{C}}_{n}^{c}$.
It now follows that for each $k<n-1$ we have \[
\frac{1}{n-k}\lambda(\mathcal{C}_{k})=\sum_{I\in{\mathcal{C}}_{k}}\frac{1}{n-k}\lambda(I)\asymp\sum_{I\in{\mathcal{C}}_{k}}\sum_{X\in\{A,B\}}\lambda(I_{X,n-k}).\]
 Combining these observations, we obtain that \[
\sum_{{k=2}}^{n-2}\frac{1}{n-k}\lambda(\mathcal{C}_{k})\asymp\sum_{k=2}^{n-2}\sum_{I\in{\mathcal{C}}_{k}}\sum_{X\in\{A,B\}}\lambda(I_{X,n-k})\leq2\lambda(\mathcal{C}_{n}^{c}).\]
 To finish the proof, let us assume by way of contradiction that $\liminf_{n\to\infty}\lambda(\mathcal{C}_{n})=\kappa>0$.
By the above, we then have that \[
1\geq\lambda(\mathcal{C}_{n}^{c})\gg
\sum_{{k=2}}^{n-2}\frac{1}{n-k}\lambda(\mathcal{C}_{k})\gg
\kappa\sum_{{k=2}}^{n-1}\frac{1}{k}\gg\log n,\mbox
{ for all }n\in\N,\]
where $a_{n}\gg b_{n}$ means that
the quotient $a_{n}/b_{n}$ is uniformly bounded away from zero.
 This gives a contradiction, and hence finishes the proof. 
\end{proof}

\noindent For the first proof of Theorem \ref{thm1} we also require the
following lemma. For this, the reader might like to recall from 
Section \ref{subsec:Farey} that
the function $\phi_{0}:\mathcal{C}_{0}\to\mathcal{C}_{0}$ is given
by $\phi_{0}(x):=x$. 
\begin{lem}
\label{lem:decrease} On $\mathcal{C}_{1}$ we have \[
\hat{T}^{n}\phi_{0}<\hat{T}^{n-1}\phi_{0},\mbox{ {\it for all} }n\in\N.\]
 \end{lem}
\begin{proof}
Recall that $\hat{T}g=\phi_{0}\cdot\mathcal{L}(g/\phi_{0})$, where
$\mathcal{L}(g)=\sum_{i=0}^{1}\left(\left|(T^{-1})'\right|\cdot(g\circ u_{i})\right)$,
that is, \[
\widehat{T}g\left(x\right)=
\frac{g\left(u_{0}(x)\right)+x\cdot g\left(u_{1}(x)\right)}{1+x}.
\]
 By \cite[Lemma 3.2]{KesseboehmerSlassi:08} it follows that for $\mathcal{D}:=\left\{ g\in C^{2}\left(\left[0,1\right]\right):g'\geq0,g''\leq0\:\right\} $
we have $\widehat{T}\left(\mathcal{D}\right)\subset\mathcal{D}$.
The latter displayed formula in particular also shows that $f\left(1/2\right)=\widehat{T}f(1)$.
Moreover, one immediately verifies that $\phi_{0}\in\mathcal{D}$. Hence,
for all $x\in\mathcal{C}_{1}$ we have \begin{eqnarray*}
\hat{T}^{n}\phi_{0}(x) & \leq & \max\left\{ \hat{T}^{n}\phi_{0}(x):x\in\mathcal{C}_{1}\right\} =\hat{T}^{n}\phi_{0}(1)=\hat{T}^{n-1}\phi_{0}\left(1/2\right)\\
 & = & \min\left\{ \hat{T}^{n-1}\phi_{0}(x):x\in\mathcal{C}_{1}\right\} \leq\hat{T}^{n-1}\phi_{0}\left(x\right).\end{eqnarray*}

\end{proof}
~ 
\begin{proof}
[First proof of Theorem \ref{thm1}] Using  Lemma \ref{Farey}, 
Lemma \ref{lem:decrease},
the $T$-invariance of $\mu$, 
and the fact that $\mathrm{d}\lambda=\phi_{0}\cdot\mathrm{d}\mu$,
we obtain \begin{eqnarray*}
\lambda(\mathcal{C}_{n+1}) & = & \mu\left(\1_{\mathcal{C}_{n+1}}\cdot\phi_{0}\right)=\mu\left(\1_{T^{-n}(\mathcal{C}_{1})}\cdot\phi_{0}\right)=\mu\left(\1_{\mathcal{C}_{1}}\cdot\hat{T}^{n}(\phi_{0})\right)\\
 & < & \mu\left(\1_{\mathcal{C}_{1}}\cdot\hat{T}^{n-1}(\phi_{0})\right)=\mu\left(\1_{\mathcal{C}_{n}}\cdot\phi_{0}\right)=\lambda(\mathcal{C}_{n}).\end{eqnarray*}
 Hence, the sequence $\left(\lambda(\mathcal{C}_{n})\right)$
is strictly decreasing. Combining this fact with Lemma \ref{Lemma0},
our first proof of Theorem \ref{thm1} is complete. 
\end{proof}

\subsection{Second Proof of Theorem \ref{thm1}}

For the second proof of 
 Theorem \ref{thm1} recall that a non-singular transformation $S$
 of the $\sigma$--finite measure space $\left(\mathcal{C}_{0}, 
 \mathcal{A}, m\right)$ 
is called exact if and only if for each element $A$ of the tail 
$\sigma$-algebra $\bigcap_{n\in \N} S^{-n}\left(\mathcal{A}\right)$
we have that $m(A) \cdot m(A^{c})=0$. Crucial for us here will be
    a result of Lin \cite{Lin} which gives a necessary and sufficient 
    condition for exactness of $S$ in terms the dual $\hat{S}$ of $S$.
  More precisely,   Lin found  that $S$ is exact if and only if 
    \[ \lim_{n  \to \infty} \|\hat{S}^{n}(f)\|_{1} = 0, \mbox{  
   for all $ f \in L_{1}(m)$ such that $ m(f)=0$}.\]
  We begin with by showing that the infinite Farey system
  $\left(\mathcal{C}_{0},T,\mathcal{A},
\mu\right)$
is exact.  
Let us remark that 
this fact is probably well known to experts in
 the field of infinite ergodic theory of numbers. 
 Nevertheless, we were unable to locate a rigorous 
proof in the literature, and hence decided to give such a proof 
here. However, our proof was inspired by the proof of 
\cite[Theorem 3.2]{ADU}.

\begin{prop}
    The Farey map $T$ of the $\sigma$--finite measure space 
    $\left(\mathcal{C}_{0},\mathcal{A},
    \mu\right)$ is exact.
    \end{prop}
\begin{proof}
    Let  $A_0	\in\bigcap_{n\in \N}T^{-n}
    \mathcal{A}$ be given such that $m_{{\mathfrak{g}}}(A_0)>0$, where 
    $d m_{{\mathfrak{g}}}(x)=(\log(2)(1+x))^{-1}d\lambda(x)$ denotes the 
    Gauss measure.  Note that, since $\mu$ and $m_{{\mathfrak{g}}}$ are 
    in the same measure class, it is sufficient to show the exactness 
    of $T$ with respect to $m_{{\mathfrak{g}}}$, rather than $\mu$.
    Therefore,  the aim is to show that $m_{{\mathfrak{g}}}(A_0^{c})=0$.
    For this, first note that, since  $A_0	\in\bigcap_{n\in \N}T^{-n}
    \mathcal{A}$,  there exists a sequence $\left(A_{n}\right)_{n\in \N}$ such 
    that $A_{n}\in\mathcal{A}$ and $A_{0}=T^{-n}A_n$, for all $n\in \N$. 
    Clearly,  we then have that $A_{k+m}=T^{k}A_m$, for all $k,m\in \N_0$.  
    For each $x \in \mathcal{C}_{0}$, let $\rho$ be defined by \[
    \rho(x):=\inf\left\{ n\geq 0:T^{n}\left(x\right)\in\mathcal{C}_{1}\right
    \} .\]
    Since $T$ is conservative, we have that $\rho$ is finite, $m_{{\mathfrak{g}}}$-almost everywhere. 
       Define $\rho_{n}:=\sum_{k=0}^{n-1}\rho\circ 
    \left(\mathfrak{g}^k\right)$, and let 
   $\langle\hspace{-.9mm}\langle x_{1},\ldots,x_{n}\rangle\hspace{-.9mm}\rangle:=\{\langle y_{1},y_{2}, \ldots 
   \rangle: y_{k}=x_{k}, k=1,\ldots,n\}$
    denote a cylinder set arising from 
the Farey coding.
 Using  the facts that 
 $m_{{\mathfrak{g}}}$  is $\mathfrak{g}$--invariant and
 of bounded mixing type with respect to  
 ${\mathfrak{g}}$,  we obtain  for 
 $m_{{\mathfrak{g}}}$--almost every 
    $x=\langle x_{1},x_{2}, \ldots \rangle =[a_1,a_2,\ldots]$, 
\begin{eqnarray}
    m_{{\mathfrak{g}}}\left(A_0	|\langle\hspace{-.9mm}\langle x_{1},\ldots,x_{\rho_{n}\left(x\right)+1}
    \rangle\hspace{-.9mm}\rangle
    \right) 
      & = & \frac{m_{{\mathfrak{g}}}\left(A_0 \cap \langle\hspace{-.9mm}\langle x_{1},\ldots,
      x_{\rho_{n}\left(x\right)+1
    }\rangle\hspace{-.9mm}\rangle\right)}{m_{{\mathfrak{g}}}\left(\langle\hspace{-.9mm}\langle x_{1},\ldots,x_{\rho_
    {n}\left(x\right)+1}\rangle\hspace{-.9mm}\rangle\right)} \nonumber\\
      & = & \frac{m_{{\mathfrak{g}}}\left(T^{-(\rho_{n}\left(x\right)+
      1)}A_{\rho_{n}\left(x
    \right)+1}\cap \langle\hspace{-.9mm}\langle 
    x_{1},\ldots,x_{\rho_{n}\left(x\right)+1}\rangle\hspace{-.9mm}\rangle
    \right)}{m_{{\mathfrak{g}}}\left(\langle\hspace{-.9mm}\langle x_{1},\ldots,x_{\rho_{n}\left(x
    \right)+1}\rangle\hspace{-.9mm}\rangle\right)} \nonumber\\
    & = & \frac{m_{{\mathfrak{g}}}\left( {\mathfrak{g}}^{-n} A_{\rho_{n}\left(x
	\right)+1}\cap 
	\langle\hspace{-.9mm}\langle x_{1},\ldots,x_{\rho_{n}\left(x\right)+1}\rangle\hspace{-.9mm}\rangle
	\right)}{m_{{\mathfrak{g}}}\left(\langle\hspace{-.9mm}\langle x_{1},\ldots,x_{\rho_{n}\left(x
	\right)+1}\rangle\hspace{-.9mm}\rangle \right)} \nonumber\\
	  &= & \frac{m_{{\mathfrak{g}}}\left( {\mathfrak{g}}^{-n} A_{\rho_{n}\left(x
	\right)+1}\cap [\hspace{-.6mm}[a_1,\ldots ,a_n]\hspace{-.6mm}]
	\right)}{ m_{{\mathfrak{g}}}\left([\hspace{-.6mm}[a_1,\ldots ,a_n]\hspace{-.6mm}]\right)} \nonumber\\
    & \asymp & \frac{m_{{\mathfrak{g}}}\left({\mathfrak{g}}^{-n}A_{\rho_{n}
    \left(x\right)+1}\right) m_{{\mathfrak{g}}} \left([\hspace{-.6mm}[a_1,\ldots ,a_n]\hspace{-.6mm}]\right)}{ m_{{\mathfrak{g}}} 
    \left([\hspace{-.6mm}[a_1,\ldots ,a_n]\hspace{-.6mm}]\right)}
       = {m_{{\mathfrak{g}}} \left(A_{\rho_{n}
    \left(x\right)+1}\right)}.
    \nonumber \end{eqnarray}
    Also, by the Martingale Convergence Theorem (cf. \cite{DOOB}), we have for 
    $m_{{\mathfrak{g}}}$-almost 
every $x=\langle x_{1},x_{2}, \ldots \rangle$,   \begin{eqnarray}
   \lim_{n\to \infty} m_{{\mathfrak{g}}}\left(A_0 | 
   \langle\hspace{-.9mm}\langle x_{1},\ldots,x_{\rho_{n}\left(x\right)+1}\rangle\hspace{-.9mm}\rangle
    \right)=\1_{A_0}(x).
    \nonumber \end{eqnarray}
Combining the  two latter observations, it follows that $A_0=\Lambda \mod m_{\mathfrak{g}}$, where
 $\Lambda$ is defined by \[\Lambda:=\{x\in \mathcal{C}_0: \liminf_{n} 
 {m_{{\mathfrak{g}}} \left(A_{\rho_{n}
    \left(x\right)+1}\right)}>0  \}.\]
Since, by  assumption,  $m_{{\mathfrak{g}}}(A_0)>0$, we now have that 
$m_{{\mathfrak{g}}}(\Lambda)>0$. Hence, to finish the proof, we are left to show that 
$m_{{\mathfrak{g}}}(\Lambda)=1$. For this recall that    $m_{{\mathfrak{g}}}$ is ergodic 
and $\mathfrak{g}$--invariant. This gives that it is in fact sufficient to show 
that ${\mathfrak{g}}^{-1}\Lambda\subset \Lambda \mod m_{{\mathfrak{g}}}$. In other words, 
in order to complete the proof, we are left to show that $\liminf_n 
m_{{\mathfrak{g}}}(A_{\rho_n({\mathfrak{g}}(x)) +1})>0$ implies 
$\liminf_n m_{{\mathfrak{g}}}(A_{\rho_n(x) +1})>0$.
Since $A_{\rho_{n+1}(x)+1}=A_{\rho(x)+\rho_{n}({\mathfrak{g}} (x))+1}=T^{\rho(x)}A_{\rho_n({\mathfrak{g}}(x)) +1}$,  
this assertion would follow if we establish that  for each 
$\epsilon>0$ and $\ell \in \N$ 
there exists $\kappa>0$ such that for all $B\in \mathcal{A}$ with
$m_{{\mathfrak{g}}}(B)>\epsilon$ we 
have $m_{{\mathfrak{g}}}(T^{\ell}B)>\kappa$. 
Hence, let us assume that $m_{{\mathfrak{g}}}(B)>\epsilon$, and let 
$\alpha_{\ell}$ 
denote the Markov partition for the map $T^{\ell}$. Clearly, 
there are $ 2^{\ell}$ elements in $\alpha_{\ell}$.
This immediately implies that  $m_{{\mathfrak{g}}}(A\cap B)> \epsilon 
2^{-\ell}$, for some $A\in \alpha_{\ell}$. 
Therefore, using the fact  that $T^{\ell}:A\to \mathcal{C}_0$ is bijective and 
the fact that there exists a constant  $c_{0}>0$ such that  $d m_{{\mathfrak{g}}}\circ 
T^{\ell}/d m_{{\mathfrak{g}}}(y) > c_{0}$ for all $y \in A$,  
it follows 
that $m_{{\mathfrak{g}}}(T^{\ell}B)> c_{0} 2^{-\ell}\epsilon$. 
Hence, by setting in the above $\kappa:=c_{0} 2^{-\ell} \epsilon$, the proof follows.
\end{proof}
\begin{prop}\label{pre-exact}
 For each $C \in \mathcal{A}$ with $\mu\left(C\right)<\infty$,
 we have that
 \[ \lim_{n\to \infty} \lambda\left( T^{-n} (C)\right) =0.\]
    \end{prop}
    \begin{proof}
Let $C\in \mathcal{A}$ be given as stated in the proposition.  
For each $A\in \mathcal{A}$ for which
$0<\mu\left(A\right)<\infty$, we then have
\begin{eqnarray*}
\lambda\left(T^{-n} (C)\right) & = & \mu\left(\mathbf{1}_{T^{-n} (C)}\, 
\cdot \phi_{0} \right)=\mu \left(\mathbf{1}_{C}\circ T^{n}\cdot 
\phi_{0}\right)
\\ & =&  \mu\left(\mathbf{1}_{C}\circ T^{n}
  \cdot\left(\phi_{0}-\frac{\mathbf{1}_{A}}{\mu
 \left(A\right)}+\frac{\mathbf{1}_{A}}{\mu\left(A\right)}\right)\right)\\ & \leq & \left\Vert \hat{T}^{n}\left(\phi_{0}-\frac{\mathbf{1}_{A}}{\mu
 \left(A\right)}\right)\right
 \Vert _{1}+\frac{\mu\left(T^{-n}\left(C\right)\cap 
 A\right)}{\mu\left(A\right)}
 \\ & \leq &\left\Vert \hat{T}^{n}\left(\phi_{0}-\frac{\mathbf{1}_{A}}{\mu
 \left(A\right)}\right)\right
 \Vert _{1}+\frac{\mu\left(C\right)}{\mu\left(A\right)} \longrightarrow
 \frac{\mu\left(C\right)}{\mu\left(A\right)}, \mbox{ 
 for $n$ tending to infinity}.
\end{eqnarray*}
Here, the latter follows, since $T$ is exact and $\mu\left( \left( 
\phi_{0}-{\1}_{A}/\mu
\left(A\right)\right)\right)=0$, and hence, Lin's criterion, mentioned 
at the beginning of this section, is applicable. 
Therefore,
by choosing $\mu\left(A\right)$  arbitrarily large,  the proposition 
follows.
\end{proof}

\begin{proof}[Second proof of Theorem \ref{thm1}] 
In Proposition \ref{pre-exact} put $C = \mathcal{C}_{1}$, and then 
use  the fact
that $\mathcal{C}_{n}=T^{-(n-1)}(\mathcal{C}_{1})$, for all $n \in \N$.
\end{proof}

\section{Proof of Theorem \ref{thm1.5}}
   \begin{proof}
	We employ several
	standard  arguments from infinite ergodic theory. First, 
	note that it is well known that   the 
	    induced map $T_{\mathcal{C}_{1}}$  of the Farey 
	    map $T$ on $\mathcal{C}_{1}$  is conjugate to 
	    the Gauss map $\mathfrak{g}$. This then immediately gives that $T_{\mathcal{C}_{1}}$
	    is continued fraction mixing (see 
	    \cite{Wirsing}). Therefore, by \cite[Lemma 3.7.4]{Aaronson:97}, it follows 
	    that $\mathcal{C}_{1}$ is a Darling--Kac set for $T$. 
	    This implies that  there  exists a sequence 
	    $\left(\nu_n\right)$ (the return
	  sequence of $T$) such that 
	  \[ \lim_{n \to \infty} \frac{1}{\nu_n} \sum_{i=0}^{n-1} \widehat{ T}^i 
	  \1_{\mathcal{C}_{1}}(x) =
	  \mu(\mathcal{C}_{1})
	  =\log 2, 
	  \mbox{ uniformly for $\mu$-almost every  } x \in \mathcal{C}_{1}.\]
	  In order to determine the asymptotic type of the sequence $(\nu_n)$, 
	  recall from \cite[Section 3.8]{Aaronson:97}
	  that for a set $C\in\mathcal{A}$ such that $0<\mu\left(C\right)<\infty$,
	  the wandering rate of $C$ is given by the sequence 
	  $\left(W_{n}\left(C\right)\right)$,
	  where  \[
	  W_{n}\left(C\right):=\mu\left(\bigcup_{k=1}^{n}T^{-(k-1)}(C)\right).\]
	  Let us  compute 
$\left(W_{n}(C)\right)$ for $C= \mathcal{C}_{1}$.  Namely, for all $n\in\N$ we have  \[
W_{n}\left(\mathcal{C}_{1}\right)=\mu\left(\bigcup_{k=1}^{n}
T^{-(k-1)}\left(\mathcal{C}_{1}\right)\right)=\mu\left(\1_{[1/(n+1),1]}\right)=
\log(n+1).\]
	   Note that  this   wandering rate
	   is  slowly varying at infinity,
	  that is  (see e.g. \cite{BinghamGoldieTeugels:89}),   \[
	  \lim_{n\to\infty}W_{k\cdot n}\left(\mathcal{C}_{1}\right)/W_{n}
	  \left(\mathcal{C}_{1}\right)=1,\mbox{ for each }k\in\N.\]
Also, note that, since $T$ has a Darling--Kac set, it follows from 
\cite[Proposition 3.7.5]{Aaronson:97} that $T$ is pointwise dual ergodic 
with respect to $\mu$, 
that is,
\[ \lim_{n\to \infty} \frac{1}{\nu_n} \sum_{i=0}^{n-1} \widehat{ 
	  T}^i f = \mu (f), \mbox{
	    for all } f\in L^1(\mu).\] 
	  In this situation we then have, by \cite[Proposition 
	  3.8.7]{Aaronson:97}, that 
	    the return
	  sequence 
	  and  the wandering rate are related 
	   through 
	   \[\lim_{n\to\infty}(n\cdot v_{n}/W_{n}(\mathcal{C}_{1}))=1.\] 
Combining these observations,  the proof of Theorem 
\ref{thm1.5} follows.
\end{proof}

\begin{rem}
Although we are not going to use these facts here, let us nevertheless remark that 
  the Farey map  $T$ has the the following additional infinite ergodic 
  theoretical properties. The 
  verification of these properties follows from standard infinite 
  ergodic theory
	    (cf. \cite{Aaronson:97},
	\cite{ADU}, \cite{Thaler:00}).

	\begin{itemize}
\item 
	  The map $T$ is rationally ergodic with respect to $\mu$. That is,
	  there exists a constant $c>0$ and a set $A$
	  with $0< \mu(A) < \infty$ such that for all $n\in \N$,
	  
 \hfill $\displaystyle
	 \int_A \bigg(\sum_{i=0}^{n-1} \1_A
	 \circ T^i\bigg)^2 d\mu < c \,   
	 \bigg(\int_A \sum_{i=0}^{n-1} \1_A \circ T^i d\mu\bigg)^2. 
	$ \hfill $(\ast)$
	\item 
	The map $T$ has the following mixing property. For $A$ with $0<\mu(A) <
	\infty$ such that  $(\ast)$ holds, we have for all $U,V \subset A$,
	\[  \lim_{n\to \infty} \frac{1}{\nu_n} \sum_{i=0}^{n-1} \mu(U \cap T^{-i} V) 
	= \mu(U)\mu(V). \]
	\end{itemize}
\end{rem}

\section{Proof of Theorem \ref{thm2}}
\begin{proof}
As already mentioned in the introduction, the proof of Theorem \ref{thm2}
will make use of some further, slightly more advanced infinite ergodic theory.
Let us begin with by first giving the concepts and results which are
relevant for the proof of Theorem \ref{thm2}. The following concept of 
a uniform set  is vital in many situations
within infinite ergodic theory, and this is also the case in our situation
here. (For further examples of interval maps (including the Farey map)
for which there exist uniform sets we refer 
to \cite{Thaler:80,Thaler:83}.)
\begin{itemize}
\item [(I)] (\cite[Section 3.8]{Aaronson:97}) \textit{A set $C\in\mathcal{A}$
with $0<\mu\left(C\right)<\infty$ is called uniform for $f\in L_{1}^{+}\left(\mu\right)$,
if  $\mu$--almost everywhere and uniformly
on $C$ we have that \[
\lim_{n\to\infty}\frac{1}{v_{n}}\sum_{k=0}^{n-1}\hat{T}^{k}\left(f\right)=\mu(f),\]
where  $\left(v_{n}\right)$ denotes the return 
sequence of $T$, and uniform convergence is meant with respect to $L_{\infty}\left(\mu|_{C}\right)$.} 
\end{itemize}
Note that it is not difficult to see that the Farey map $T$ satisfies
Thaler's conditions, among which Adler's condition, i.e. $T''/(T')^{2}$
is bounded throughout $(0,1)$, is the most important one (see \cite{Thaler:80,Thaler:83}).
This then immediately implies that we have the following, where, as
in Section \ref{subsec:Farey}, the function $\phi_{0}$ is given
by $\phi_{0}(x)=x$. 
\begin{itemize}
\item [(II)] {\em Let $C\in\mathcal{A}$ be given with $\lambda\left(C\right)>0$
and so that there 
exists an $\epsilon>0$ such that  $x>\epsilon$, for all $x \in C$.
We then have that $C$
is a uniform set for the function $\phi_{0}$. } 
\end{itemize}
Now, the crucial notion for proving the sharp asymptotic result of
Theorem \ref{thm2} is provided by the following concept of a
uniformly returning set. (For further examples of one dimensional
dynamical systems which allow uniformly returning sets for some appropriate
function we refer to \cite{Thaler:00}.) 
\begin{itemize}
\item [(III)] (\cite{KesseboehmerSlassi:05}) \textit{A set $C\in\mathcal{A}$
with $0<\mu\left(C\right)<\infty$ is called uniformly returning for
$f\in L_{\mu}^{+}$ if there exists a positive increasing sequence
$\left(w_{n}\right)=\left(w_{n}(f,C)\right)$ of positive reals such
that $\mu$--almost everywhere and uniformly on $C$ we have 
\[
\lim_{n\to\infty}w_{n}\hat{T}^{n}\left(f\right)=\mu(f).\]
}  
\end{itemize}
In order to determine the asymptotic type of  the sequence 
$\left(w_{n}\right)$,  we use \cite[Proposition 
1.2]{KesseboehmerSlassi:05} where we found that \[
\lim_{n\to\infty}W_{n}\left(C\right)/w_{n}=1,\mbox{ for all $C\in\mathcal{A}$ 
such that $0<\mu\left(C\right)<\infty$},\]
where $\left(W_{n}\left(C\right)\right)$ denotes the wandering rate, 
which we already considered in the proof of Theorem 
\ref{thm1.5}.  
In \cite[Proposition 1.1]{KesseboehmerSlassi:05}
it was shown that every uniformly returning set is uniform. Whereas,
in \cite{KesseboehmerSlassi:08} we found explicit conditions under
which also the reverse of this implication holds. Applying these results
of \cite{KesseboehmerSlassi:08} to our situation here, one obtains
the following. 
\begin{itemize}
\item [(IV)] (\cite{KesseboehmerSlassi:08}) {\em Let $C\in\mathcal{A}$
with $0<\mu\left(C\right)<\infty$ be a uniform set, for some $f\in L_{\mu}^{+}$.
If the wandering rate $\left(W_{n}(C)\right)$ is slowly varying at
infinity and if the sequence $\left(\hat{T}^{n}\left(f\right)\mid_{C}\right)$
is decreasing, then we have that $C$ is a uniformly returning set
for $f$. Moreover, $\mu$--almost everywhere and uniformly on $C$
we have  \[
\lim_{n\to\infty}W_{n}(C)\,\hat{T}^{n}\left(f\right)=\mu(f).\]
 } 
\end{itemize}
With these preparations, we can now finish  the proof of Theorem 
\ref{thm2} as follows. 
 The idea is to apply the results
stated above to the situation in which the set $C$ is equal to $\mathcal{C}_{1}$.
For this, first recall that we have already seen that the wandering rate $\left(W_{n}(\mathcal{C}_{1})\right)$
of $\mathcal{C}_{1}$ is obviously slowly varying at infinity. In
fact, as computed in the proof of Theorem 
    \ref{thm1.5}, we have that $\lim_{n\to\infty}n\cdot 
    v_{n}/W_{n}(\mathcal{C}_{1})=1$, and also that $W_{n}(\mathcal{C}_{1}) 
    \sim \log n$. 
 Secondly, since $\mathcal{C}_{1}$ is bounded away from zero, the
result in (II) gives that $\mathcal{C}_{1}$ is a uniform set for
$\phi_{0}$. Thirdly, by Lemma \ref{lem:decrease}, we have that the
sequence $\big(\hat{T}^{n}\left(\phi_{0}\right)\mid_{\mathcal{C}_{1}}\big)$
is decreasing. Thus, we can apply the result in the first part of  (IV), which then
shows that $\mathcal{C}_{1}$ is a uniformly returning set for the
function $\phi_{0}$. Hence, the second part in (IV) gives that
$\mu$--almost everywhere and uniformly on $\mathcal{C}_{1}$ we have\[
\lim_{n\to\infty}W_{n}(\mathcal{C}_{1})\,\hat{T}^{n}\left(\phi_{0}\right)=\mu(\phi_{0})=1.\]
 Combining these observations, it now follows that \[
\lim_{n\to\infty}\left(\log n\cdot\lambda\left(\mathcal{C}_{n}\right)\right)=\lim_{n\to\infty}\left(W_{n}(\mathcal{C}_{1})\cdot\mu\left(\1_{\mathcal{C}_{1}}\cdot\widehat{T}^{n-1}(\phi_{0})\right)\right)=\mu\left(\1_{\mathcal{C}_{1}}\right)=\log2.\]
 This finishes the proof of Theorem \ref{thm2}. 
\end{proof}
\section{Thermodynamical significance of the sum-level 
sets}\label{Thermo}

In this section we discuss the thermodynamical significance
of the results of the previous sections. For this, recall that in \cite{KesseboehmerStratmann:04A} and \cite{KesseboehmerStratmann:07}
(see also \cite{KesseboehmerStratmann:04B}) we studied
the multifractal spectrum
$\{\tau(s):s\in\R\}$, given by \[
\tau(s):=\dim_{H}\left(\left\{ x=[a_{1},a_{2},\ldots]:\lim_{n\rightarrow\infty}
\frac{2\log q_{n}(x)}{\sum_{i=1}^{n}a_{i}}=s\right\} \right).\]
 Here, $p_{n}(x)/q_{n}(x):=[a_{1},a_{2},\ldots,a_{n}]$ denotes the
$n$-th approximant of $x$, and $\dim_{H}$ refers to the Hausdorff
dimension. In order to compute this spectrum, the Stern--Brocot pressure
function $P$ turns out to be crucial. This pressure function is defined
for $t\in\R$ by \[
P(t):=\lim_{n\rightarrow\infty}\frac{1}{n}\log\sum_{I 
\in\mathcal{T}_{n}}\left(\diam(I)\right)^{t}.\]
 The following results give the main outcome concerning the 
properties of $P$ and $\tau$ and on how these functions are related.
This complete thermodynamical description of the Stern--Brocot system
was obtained in \cite[Theorem   1.1]{KesseboehmerStratmann:07}. Here,
$\gamma:=(1+\sqrt{5})/2$ denotes  the Golden Mean, and $P^{*}$
refers to the Legendre transform of $P$, given for $s\in\R$
by $P^{*}(s):=\sup_{t\in\R}\{t\cdot s-P(t)\}$. 
\begin{enumerate}
\item \textit{\cite[Theorem 1.1]{KesseboehmerStratmann:07}. For each $s\in[0,2\log\gamma]$,
we have that \[
\tau(s)=-P^{*}(-s)/s,\]
 with the convention that $\tau(0):=\lim_{s\searrow0}-P^{*}(-s)/s=1$.
Also, the dimension function $\tau$ is continuous and strictly decreasing
on $[0,2\log\gamma]$ and vanishes outside the interval $[0,2\log\gamma)$.
Moreover, the left derivative of $\tau$ at $2\log\gamma$ is equal
to $-\infty$.} 
 \textit{ The function $P$ is convex, non-increasing and differentiable
throughout $\R$. Furthermore, $P$ is real-analytic on $(-\infty,1)$
and vanishes on $[1,\infty)$.}  
\item \textit{\cite{HR,PS} \label{HRPS} We have that
\[
P(1-\epsilon)\sim-\epsilon/\log\epsilon, \mbox{ for $\epsilon$ tending 
to zero form above}.\]
 In particular, the Farey system has a second order phase transition
at $t=1$, that is, the function $P'$ is continuous and $P''$ is
discontinuous at $t=1$.}
\end{enumerate}
The following shows that the vanishing of $\lim_{n\to\infty}\lambda(\mathcal{C}_{n})$
is very much a phenomenon of the fact that the Stern--Brocot system
exhibits a phase transition of order two at $t=1$. At this point
of intermittency,  finite ergodicity breaks down and infinite ergodic
theory enters the scene. In particular, by (\ref{HRPS}), this abrupt transition from
finite to infinite ergodic theory happens in a way which is non-smooth.

One aspect of this intermittency  is given by the following.  For 
this, recall that  in \cite[Proposition 2.6]{KesseboehmerStratmann:04A} it 
was also shown that for each $s \in 
 (0,2\log\gamma]$ there exists an 
equilibrium measure $\nu_{s}$ for which $\dim_{H}(\nu_{s}) = \tau(s)$. Using 
the invariance  of $\nu_{s}$,   one immediately verifies that 
for each $s\in(0,2\log\gamma]$  we have that
 \[
\nu_{s}(\mathcal{C}_{1})=
\nu_{s}(\mathcal{C}_{n})=
\nu_{s}(\mathcal{C}_{n}^{c})=1/2 ,\mbox{ for all } n\in \N.\]
 In contrast to this, we have by Theorem \ref{thm1} and Theorem \ref{thm2} respectively,
\[
\lim_{n\rightarrow\infty}\lambda(\mathcal{C}_{n})=0
\mbox{ and }\lim_{n\rightarrow\infty}\lambda(\mathcal{C}_{n}^{c})
=1.\]

Another  aspect is provided by the following proposition, in which 
$P_{0}$
and  $P_{1}$
denote  the two partial
Stern--Brocot pressure functions, given  for $t\in\R$ by \[
P_{0}(t):=\lim_{n\rightarrow\infty}\frac{1}{n}\log\sum_{I\in
\mathcal{C}_{n}}\left(\diam(I)\right)^{t}\mbox{ and }P_{1}(t):=
\lim_{n\rightarrow\infty}\frac{1}{n}\log\sum_{I\in\mathcal{C}_{n}^{c}}
\left(\diam(I)\right)^{t}.\]

\begin{prop}
The outcome of the above complete thermodynamical description 
of the
Stern--Brocot system stays to be the same if we base the analysis 
exclusively on
either $\{ I \in \mathcal{C}_{n}: n \in \N\}$ or $\{I \in 
\mathcal{C}_{n}^{c}: n \in \N\}$, rather than on
all the intervals in $ \{ I \in \mathcal{T}_{n}: n \in \N\}$. In particular, we have that
\[
P(t)=P_{0}(t)=P_{1}(t),\mbox{ for all }t\in\R.\]
\end{prop}
\begin{proof}
 Using the recursive definition of the Stern--Brocot sequence,
one immediately verifies that \[
t_{n-1,2k-1}\, t_{n-1,2k}\leq t_{n,4k-2}\, t_{n,4k-1}\leq n\,\, t_{n-1,2k-1}\, t_{n-1,2k},\]
 and \[
t_{n-1,2k}\, t_{n-1,2k+1}\leq t_{n,4k-1}\, t_{n,4k}\leq n\,\, t_{n-1,2k}\, t_{n-1,2k+1}.\]
 Combining these observations, we obtain \[
n^{-|t|}\,\sum_{I\in\mathcal{T}_{n-1}}\left(\diam(I)\right)^{t}
\leq\sum_{I\in\mathcal{C}_{n}}\left(\diam(I)\right)^{t}
\leq n^{|t|}\,\sum_{I\in\mathcal{T}_{n-1}}\left(\diam(I)\right)^{t}.\]
 This shows that $P(t)=P_{0}(t)$, for all $t\in\R$. The proof of
$P(t)=P_{1}(t)$ follows by similar means, and is left to the reader. \end{proof}
\begin{rem}
Note that Feigenbaum, Procaccia and Tél \cite{FPT} explored what they
called
the Farey tree model.  This model is based on the set of even intervals
\[
\left\{ \left[\frac{s_{n,4k-2}}{t_{n,4k-2}},\frac{s_{n,4k}}{t_{n,4k}}\right]:
k=1,\ldots,2^{n-2}\right\} ,\mbox{ for all }n\in\N.\]
Note that for the even intervals of any order $n\in \N$ we have, for all $t\in\R$, 
\begin{eqnarray*}
 & \, & \hspace{-14mm}\left(\diam\left(\left[\frac{s_{n,4k-2}}{t_{n,4k-2}},
 \frac{s_{n,4k}}{t_{n,4k}}\right]\right)\right)^{t}\\
 & = & \left(\diam\left(\left[\frac{s_{n,4k-2}}{t_{n,4k-2}},\frac{s_{n,4k-1}}{t_{n,4k-1}}\right]\right)+\diam\left(\left[\frac{s_{n,4k-1}}{t_{n,4k-1}},\frac{s_{n,4k}}{t_{n,4k}}\right]\right)\right)^{t}\\
\hspace{-7cm} & \asymp & \left(\diam\left(\left[\frac{s_{n,4k-2}}{t_{n,4k-2}},\frac{s_{n,4k-1}}{t_{n,4k-1}}\right]\right)\right)^{t}+\left(\diam\left(\left[\frac{s_{n,4k-1}}{t_{n,4k-1}},\frac{s_{n,4k}}{t_{n,4k}}\right]\right)\right)^{t}\\
 & = & \left(t_{n,4k-2}\, t_{n,4k-1}\right)^{-t}+\left(t_{n,4k-1}\, t_{n,4k}\right)^{-t}.\end{eqnarray*}
Hence, the pressure function arising from the Farey tree model 
coincides with the pressure function $P_{0}$.\end{rem}
 
 \section{Some Diophantine applications}\label{Diophant}
 Let us end the paper by  giving an interesting immediate  application 
 of 
 Theorem  \ref{thm2} to elementary metrical Diophantine analysis.  
 For this,  first recall the following well known result of 
 Khintchine 
 (see e.g. 
 \cite{Khintchine}), which states that  
 \[
 \limsup_{n\to\infty}\frac{\log\left(a_{n}/ n 
 \right)}{\log \log n} =1, \mbox{ for $\lambda$-almost every
 $[a_{1},a_{2},\ldots]$}.\]
 
 In contrast to this well known  Khintchine law,  Theorem  \ref{thm2} 
 now gives rise 
 to  the following algebraic Khintchine-like law. ( For some further 
results on the statistics of the sum of the first  continued fraction digits we refer 
to \cite{Hensley:00}.)
 \begin{lem}\label{CorD} 
 We have that  \[
 \limsup_{n\to\infty}\frac{\log\left(a_{n+1}/\sum_{i=1}^{n}a_{i}
 \right)}{\log \log\left(\sum_{i=1}^{n}a_{i}\right)}\leq0, \mbox{ for 
 $\lambda$-almost every $[a_{1},a_{2},\ldots]$}.\]
 \end{lem}
 \begin{rem}
    We choose to call the latter result algebraic Khintchine-like law,
   since  $\sum_{i=1}^{n}a_{i}$ represents   the  word 
   length 
     associated with Farey system,  whereas the parameter $n$ 
represents the word length
     associated with the Gauss system.
     \end{rem}
 \begin{proof}
  For each $n \in \N$ and $\epsilon >0$, let
  \[E_{n}^{\epsilon}:= \bigcup_{k\in \N} \left\{[\hspace{-.6mm}[a_{1},\ldots,a_{k+1}]\hspace{-.6mm}]: 
  \sum_{i=1}^{k} a_{i} =n, a_{k+1} \geq n (\log n)^{\epsilon} 
  \right\} ,\]
  and define
  \[\mathcal{E}_{n}^{\epsilon}:= \bigcup_{I\in 
  E_{n}^{\epsilon}} I.\]
 Then note that a routine calculation for the Lebesgue measure of 
 continued fraction cylinder sets gives, for all $k,\ell \in \N$,  
  \[ \sum_{a_{k+1} \geq  \ell}\lambda([\hspace{-.6mm}[a_{1},\ldots,a_{k},a_{k+1}]\hspace{-.6mm}]) 
  \asymp \ell^{-1} \lambda([\hspace{-.6mm}[a_{1},\ldots,a_{k}]\hspace{-.6mm}]).\]
  Using this estimate and Theorem  \ref{thm2}, we obtain
  \begin{eqnarray*}
     \lambda(\mathcal{E}_{n}^{\epsilon}) &=&  \sum_{k=1}^{n} 
     \sum_{(a_{1},\ldots,a_{k})\atop
     \sum_{i=1}^{k} a_{i} =n} \sum_{a_{k+1}\geq n (\log n)^{\epsilon} }
     \lambda([\hspace{-.6mm}[a_{1},\ldots,a_{k},a_{k+1}]\hspace{-.6mm}])  
     \asymp \sum_{k=1}^{n} 
     \sum_{(a_{1},\ldots,a_{k})\atop
     \sum_{i=1}^{k} a_{i} =n} 
     \frac{\lambda([\hspace{-.6mm}[a_{1},\ldots,a_{k}]\hspace{-.6mm}])}{ n (\log n)^{\epsilon} }\\
     &=& \left(n (\log n)^{\epsilon} \right)^{-1} \sum_{k=1}^{n} 
     \sum_{(a_{1},\ldots,a_{k})\atop
     \sum_{i=1}^{k} a_{i} =n} 
     \lambda([\hspace{-.6mm}[a_{1},\ldots,a_{k}]\hspace{-.6mm}])
 =   \left(n (\log n)^{\epsilon}\right)^{-1} 
 \lambda(\mathcal{C}_{n}) \\ &  \sim &
 \frac{\log 2}{ n (\log n)^{1+\epsilon}}.
        \end{eqnarray*}
 A straight forward application of the Borel-Cantelli Lemma then gives 
 that 
 \[ \lambda\left(\limsup_{n}  \mathcal{E}_{n}^{\epsilon}\right) =0, \mbox{ for each 
 $\epsilon>0$}.\]
 Hence, by considering  the complement of $\limsup_{n}  
 \mathcal{E}_{n}^{\epsilon}$ in $\mathcal{C}_{0}$, we have now shown 
 that, for each $\epsilon>0$ and 
  for $\lambda$-almost all $[a_{1},a_{2},\ldots]$,
 \[ a_{k+1} <  \left( \sum_{i=1}^{k} a_{i} \right)  \left(\log 
 \sum_{i=1}^{k} a_{i}\right)^{\epsilon}  
, \mbox{for 
 all $k\in \N$ sufficiently large}.\]
 By taking logarithms on both sides of the latter inequality,  the lemma
 follows.
 \end{proof}
 \begin{rem} Let us  remark that, in addition to the statement in 
 Lemma \ref{CorD}, we also have that  \[
  \liminf_{n\to\infty}\frac{a_{n+1}}{\sum_{i=1}^{n}a_{i}}=0, \mbox{ 
  for $\lambda$-almost 
 all $[a_{1},a_{2},\ldots]$}.\]
 This follows, since by 
 \cite[Theorem 1.1 (4)]{KesseboehmerSlassi08b}  one has that,
 for each $\epsilon>0$, \[
 \lambda\left(\left\{x=[a_{1},a_{2},\ldots]: \frac{a_{\theta_{n}(x)+1}}{\sum_{k=
 1}^{\theta_{n}(x)}
 a_{k}}>\epsilon,\;\theta_{n}(x)>0 \right\}\right)\sim\frac{{\epsilon}^{-1}
 \log\left(1+\epsilon\right)+\log\left(1+\epsilon^{-1}\right)}{\log
 n},\]
 where we have put  $\theta_{n}\left(\left[a_{1},a_{2},\ldots\right]\right):=
 \max\left\{ k \in\mathbb{N}_0\;:\;\sum_{i=1}^{k}a_{i}\leq n\right\} $.
 Therefore,   for each $\epsilon>0$ and for
 $\lambda$-almost every $[a_{1},a_{2},\ldots]$ we have that there 
 exists an increasing sequence $(n_{k})_{k\in \N}$ of positive 
 integers such that \[
 a_{n_{k}+1} \leq \epsilon \sum_{i=1}^{n_{k}}a_{i}, \mbox{ for all $k \in \N$}.\]

\end{rem}
\singlespacing

\end{document}